\documentclass[11pt]{amsart}

\usepackage{latexsym}
\usepackage{float}
\usepackage{amssymb}
\usepackage{amsmath}
\usepackage{mathrsfs}
\usepackage{euscript}
\usepackage[all]{xy}
\newcommand{\gb}{\beta}
\newcommand{\ga}{\alpha}

\newcommand{\gd}{\delta}
\newcommand{\gD}{\Delta}
\renewcommand{\ge}{\epsilon}

\newcommand{\fg}{{\mathfrak g}}
\newcommand{\fh}{{\mathfrak h}}

\newcommand{\fl}{{\mathfrak l}}

\newcommand{\fn}{{\mathfrak n}}

\newcommand{\fq}{{\mathfrak q}}

\newcommand{\fz}{{\mathfrak z}}
\newcommand{\f}{\mathfrak}

\newcommand{\eC}{\EuScript{C}}
\newcommand{\eD}{\EuScript{D}}

\newcommand{\eM}{\EuScript{M}}

\newcommand{\eV}{\EuScript{V}}

\newcommand{\eX}{\EuScript{X}}

\newcommand{\nbar}{\bar{n}}                 
\newcommand{\Nbar}{\bar{N}}

\newcommand{\R}{\mathbb{R}}          
\newcommand{\C}{\mathbb{C}}          

\newtheorem{Thm}[equation]{Theorem}
\newtheorem{Lem}[equation]{Lemma}
\newtheorem{Cor}[equation]{Corollary}
\newtheorem{Prop}[equation]{Proposition}

\numberwithin{equation}{section}
\newcommand{\be}{\begin{equation}}
\newcommand{\beu}{\begin{equation*}}

\newcommand{\acts}{ {\raisebox{1pt} {$\scriptstyle \bullet$} } }
\newcommand{\ad}{\text{ad}}
\newcommand{\Ad}{\text{Ad}}
\newcommand{\Cal}{\mathcal}

\newcommand{\IP}[2]{\langle#1 , #2\rangle}     

\newcommand{\tOmega}{\tilde{\Omega}}

\begin{document}

\bibliographystyle{amsplain}

\baselineskip=16pt

\title[Conformally Invariant Systems]
{Systems of Third-Order Differential Operators Conformally Invariant under 
$\f{sl}(3,\mathbb{C})$ and $\f{so}(8,\mathbb{C})$}

\author{Toshihisa Kubo}
\address{Department of Mathematics,
             Oklahoma State University,
              Stillwater, Oklahoma 74078}
\email{ toskubo@math.okstate.edu}


\begin{abstract}
In earlier work, Barchini, Kable, and Zierau constructed a number of conformally
invariant systems of differential operators associated to Heisenberg parabolic
subalgebras in simple Lie algebras. The construction was systematic,
but the existence of such a system was left open in two cases,
namely, the $\Omega_3$ system for type $A_2$ and type $D_4$.
Here, such a system is shown to exist for both cases.
The construction of the system may also be interpreted as 
giving an explicit homomorphism between generalized Verma modules.
\end{abstract}
\maketitle

\section{Introduction}\label{Section1}


Conformally invariant systems of differential operators 
on a smooth manifold $M$ on which a Lie algebra $\fg$ acts by first order 
differential operators were studied by Barchini, Kable, and Zierau 
in \cite{BKZ08} and \cite{BKZ09}.  
We first recall the definition of the conformally invariant systems
from \cite{BKZ09} here.
Let $\fg_0$ be a real Lie algebra. We say that a smooth manifold $M$
is a \emph{$\fg_0$-manifold} if there exists a $\fg_0$-homomorphism 
$\Pi_M: \fg_0 \to C^\infty(M)\oplus \eX(M)$, 
where $\eX(M)$ is the space of smooth vector fields on $M$. 
Given $\fg_0$-manifold $M$, we write 
$\Pi_M(X) = \Pi_0(X) + \Pi_1(X)$ with $\Pi_0(X) \in C^\infty(M)$
and $\Pi_1(X) \in \eX(M)$. 
Let $\mathbb{D}(\eV)$ denote the space of differential operators 
on a vector bundle $\eV \to M$.
Then we say that a vector bundle $\eV \to M$
a \emph{$\fg_0$-bundle} if 
there exists a $\fg_0$-homomorphism 
$\Pi_\eV: \fg_0 \to \mathbb{D}(\eV)$ so that 
in $\mathbb{D}(\eV)$ $[\Pi_\eV(X), f]= \Pi_1(X) \acts f$
for all $X \in \fg_0$ and all $f \in C^\infty(M)$,
where the dot $\acts$ denotes 
the action of the differential operator $\Pi_1(X)$. 
Note that we regard any smooth functions $f$ on $M$ 
as elements in $\mathbb{D}(\eV)$ by identifying them with
the multiplication operator they induce.
Then, given $\fg_0$-bundle $\eV\to M$,
a list of differential operators
$D_1, \ldots, D_m \in \mathbb{D}(\eV)$
is said to be a \emph{conformally invariant system}
on $\eV$ with respect to $\Pi_\eV$
if the following two conditions are satisfied:
\begin{enumerate}
\item[(S1)] The list $ D_1, \ldots, D_m $ is linearly independent at each point of $M$. 
\item[(S2)] For each $Y \in \fg_0$ there is an $m \times m$ matrix $C(Y)$ of smooth functions
on $M$ so that
\begin{equation*}
[\Pi_\eV(Y), D_j] = \sum_{i=1}^m C_{ij}(Y)D_i
\end{equation*}
\end{enumerate}
in $\mathbb{D}(\eV)$. 
By extending the $\fg_0$-homomorphisms $\Pi_M$ and $\Pi_\eV$ $\C$-linearly, 
the definitions of a $\fg_0$-manifold, a $\fg_0$-bundle, and 
a conformally invariant system can be applied equally well to
the complexified Lie algebra $\fg = \fg_0 \otimes_\R \C$.

While a general theory of conformally invariant systems is developed in \cite{BKZ09},  
examples of such systems of differential operators 
associated to the Heisenberg parabolic subalgebras
of any complex simple Lie algebras are constructed in \cite{BKZ08}.
The purpose of this paper is to answer a question, left open in 
\cite{BKZ08}, concerning the existence of certain conformally invariant
systems of third-order differential operators. 
This is done by constructing the required systems.

This result may be interpreted as giving an explicit homomorphism between
two generalized Verma modules, one of which is non-scalar.
See Section 6 of \cite{BKZ09} for the general theory.
In this paper we describe it explicitly in a less general setting
(see the discussion after Lemma \ref{Lem43}). 
The problem of constructing and classifying homomorphisms between 
scalar generalized Verma modules has received a lot of attention;
for recent work, see, for example, \cite{Matumoto06}.
In \cite{Matumoto06}, Matumoto classifies 
the non-zero $\Cal{U}(\fg)$-homomorphisms 
between scalar generalized Verma modules associated to
maximal parabolics of non-Hermitian symmetric type. With the earlier work
\cite{Boe85} of Boe for the Hermitian symmetric cases, the existence
problem for scalar generalized Verma modules of maximal parabolic type
is solved. 
However, much less is known about maps between generalized Verma modules 
that are not necessarily scalar.


In order to explain our main results in this paper,
we briefly review the results of \cite{BKZ08} here. 
To begin with, 
let $\fg$ be a complex simple Lie algebra and $\fq = \fl \oplus \fn$
be the parabolic subalgebra of Heisenberg type; 
that is,
$\fn$ is a two-step nilpotent algebra with one-dimensional center. 
We denote by $\gamma$ the highest root of $\fg$.
For each root $\ga$ let $\{X_{-\ga}, H_{\ga}, X_\ga\}$
be a corresponding  $\f{sl}(2)$-triple,
normalized as in Section 2 of \cite{BKZ08}.
Then $\ad(H_\gamma)$ on $\fg$ has eigenvalues $-2$, $-1$, $0$, $1$, $2$, 
and the corresponding eigenspace decomposition of $\fg$ is denoted by
\begin{equation}\label{Intro1}
\fg = \fz(\bar\fn) \oplus V^- \oplus \fl \oplus V^+ \oplus \fz(\fn).
\end{equation}

Let $\mathbb{D}[\fn]$ be the Weyl algebra of $\fn$;
that is, the algebra of partial differential operators on $\fn$ with polynomial coefficients.
Then each system of $k$-th order differential operators 
constructed in \cite{BKZ08} derives from a $\C$-linear map 
$\Omega_k: \fg(2-k) \to \mathbb{D}[\fn]$ 
with $1\leq k \leq 4$ and $\fg(2-k)$ the $2-k$ eigenspace of $\ad(H_\gamma)$.
Let $\Pi_s: \fg \to \mathbb{D}[\fn]$
be the Lie algebra homomorphism constructed in Section 4 in \cite{BKZ08}.
Here $s$ is a complex parameter, which indexes line bundles $\Cal{L}_{-s}$
over a real flag manifold $G_0/\bar{Q}_0$,
where $G_0$ is a real Lie group with Lie algebra $\fg_0$
and $\bar{Q}_0$ is a parabolic subgroup of $G_0$
with complexified Lie algebra $\bar \fq$ opposite to $\fq$.
We say that the $\Omega_k$ system has special value $s_0$
when the system is conformally invariant under $\Pi_{s_0}$.

In \cite{BKZ08} the special values of $s$ are determined
for the $\Omega_k$ systems with $k=1,2,4$
for all complex simple Lie algebras, but 
only exceptional cases are considered for the $\Omega_3$ system.
A table in Section 8.10 in \cite{BKZ08} lists the special values of 
$s$. The reader may want to notice that
the entries in the columns for the systems 
$\Omega_2^{\text{big}}$ and $\Omega_2^{\text{small}}$ 
for types $B_r$ and $C_r$ should be transposed.
Theorem 21 in \cite{BKZ09} then shows that the $\Omega_3$ system
does not exist for $A_r$ with $r \geq 3$, $B_r$ with $r\geq 3$, and 
$D_r$ with $r \geq 5$.
There remain two open cases, namely, the $\Omega_3$ system
for type $A_2$ and type $D_4$. 
The aim of this paper is to show 
that the $\Omega_3$ system does exist for both cases
(see Theorem \ref{Thm51} and Theorem \ref{Thm311}).
In order to achieve the result we use several facts 
from both \cite{BKZ08} and \cite{BKZ09}.
By using these facts, we significantly 
reduce the amount of computation to show the existence of the system.

There are two differences between our conventions here and those used in \cite{BKZ08}.
One is that we choose the parabolic $Q_0 = L_0 N_0$ for the real flag manifold,
while the opposite parabolic $\bar Q_0 = L_0 \bar N_0$ is chosen in \cite{BKZ08}.
Because of this, our special values of $s$
are of the form $s = -s_0$, where $s_0$ 
are the special values shown in Section 8.10 in \cite{BKZ08}.
The other is that we identify $(V^+)^*$ with $V^-$ by using the Killing form,
while $(V^+)^*$  in \cite{BKZ08} is identified with $V^+$
by using the non-degenerate alternating form $\IP{\cdot}{\cdot}$ on $V^+$
defined by $[X_1, X_2] = \IP{X_1}{X_2}X_\gamma$ for $X_1, X_2 \in V^+$. 
Because of this difference the right action $R$, which will be defined
in Section 2, will play the role played by $\Omega_1$ in \cite{BKZ08}.
In addition to these notational differences, there are also some methodological 
differences between \cite{BKZ08} and what we do here. 
These stem from the fact that we make systematic use of the results
of \cite{BKZ09} to streamline the process of proving conformal invariance. 

We now outline the remainder of this paper.
In Section \ref{Section3}, we review the setting and results of Section 5 in \cite{BKZ09},
simultaneously specializing them to the situation considered here.
In Section \ref{Section4},
we specialize further by taking $\fg$ to be simply laced.
We fix a suitable Chevalley basis
and define the $\Omega^t_3$ system
by $\Omega_3^t = \tOmega_3 + tC_3$ for $t \in \C$.
A remark on notation might be helpful here. In \cite{BKZ08},
a system $\Omega'_3$ is initially defined. It is then shown to decompose
as a sum of a leading term $\tOmega_3$ and a correction term $C_3$. 
These two are recombined with different coefficients to give $\Omega_3$,
which is finally shown to be conformally invariant for exceptional algebras.
Thus, the $\Omega_3$ system is defined to exist if there exists $t_0 \in \C$
so that the $\Omega_3^{t_0}$ system is conformally invariant.

In Section \ref{Section5}, we take $\fg$ to be of type $A_2$
and show that the $\Omega_3$ system(s) exists over the line bundle $\Cal{L}_0$.
The Heisenberg parabolic subalgebra coincides 
with the Borel subalgebra in this case.
Thus $V^-$ decomposes as the direct sum of two one-dimensional
$\fl$-submodules. This implies that there will be two $\Omega_3$ systems,
each of the operators will be conformally invariant all by itself.
The conformal invariance of the these operators is 
shown in Theorem \ref{Thm51}.

In Section \ref{Section6}, we take $\fg$ to be of type $D_4$.
For type $D_4$, the data in p.831 and Theorem 6.1 of \cite{BKZ08}
suggest that the complex parameter $t_0$ for the $\Omega_3^t$ system to 
be conformally invariant is $t_0 =0$, 
so that the correction term $C_3$ is discarded completely. 
For this reason, we simply proceed to show that $\tOmega_3$ 
is conformally invariant. This is done in Theorem \ref{Thm311}.

\vskip 0.1in

\noindent \textbf{Acknowledgment.} The author would like to
be grateful to Dr. Anthony Kable for his valuable suggestions and comments on this paper.
He would also like to thank the referee for the fruitful comments.




\section{A specialization of the theory}\label{Section3}

The purpose of this section is to 
introduce the $\fg$-manifold and the $\fg$-bundle
that we study in this paper. 
Let $G_0$ be a connected real semisimple Lie group with Lie algebra $\fg_0$ and 
complexified Lie algebra $\fg$. Let $Q_0$ be a parabolic subgroup of $G_0$ and 
$Q_0 = L_0N_0$ a Levi decomposition of $Q_0$. By the Bruhat decomposition,
the subset $\bar{N}_0Q_0$ of $G_0$ is open and dense in $G_0$,
where $\bar{N}_0$ is the nilpotent subgroup of $G_0$ opposite to $N_0$.
Let $\bar\fn$ and $\fq$ be the complexifications of 
the Lie algebras of $\bar N_0$ and $Q_0$, respectively;
we have the direct sum $\fg = \bar \fn \oplus \fq$.
For $Y \in \fg$, write $Y = Y_{\bar \fn} + Y_{\fq}$ for the decomposition 
of $Y$ in this direct sum.
Similarly, write the Bruhat decomposition of $g \in \bar N_0 Q_0$ as 
$g= \mathbf{\bar n}(g)\mathbf{q}(g)$ with $\mathbf{\bar n}(g) \in \bar N_0$ and 
$\mathbf{q}(g) \in Q_0$. Note that for $Y \in \fg_0$ we have
\beu
Y_{\bar \fn} = \frac{d}{dt} \mathbf{\bar n}(\exp(tY)) \big|_{t=0},
\end{equation*}
and a similar equality holds for $Y_{\fq}$.

We consider the homogeneous space $G_0/Q_0$. Let $\C_{\chi^{-s}}$ be 
the one-dimensional representation of $L_0$ with character $\chi^{-s}$
with $ s \in \C$, where $\chi$ is a real-valued character of $L_0$.
The representation $\chi^{-s}$ is extended to a representation of $Q_0$
by making it trivial on $N_0$.
For any manifold $M$, denote by $C^\infty(M,\C_{\chi^{-s}})$ the smooth
functions from $M$ to $\C_{\chi^{-s}}$. The group $G_0$ acts on the space
\beu
C^\infty_{\chi}(G_0/Q_0, \C_{\chi^{-s}}) = \{ F \in C^\infty(G_0, \C_{\chi^{-s}}) \; |\;
\text{$F(gq) = \chi^{-s}(q^{-1})F(g)$ for all $q \in Q_0$ and $g \in G_0$} \}
\end{equation*}
\vskip 0.1in
\noindent by left translation,
and the action $\Pi_s$ of $\fg$ on $C^\infty_\chi(G_0/Q_0, \C_{\chi^{-s}})$
arising from this action is given by
\begin{equation}\label{Eqn31}
(\Pi_s(Y) \acts F)(g) = \frac{d}{dt}F(\exp(-tY)g)\big|_{t=0}
\end{equation}
\vskip 0.1in
\noindent for $Y \in \fg_0$, where the dot $\bullet$ denotes the action of $\Pi_s(Y)$. 
This action is extended $\C$-linearly
to $\fg$ and then naturally to the universal enveloping algebra $\Cal{U}(\fg)$.
We use the same symbols for the extended actions.

The restriction map 
$C^\infty_{\chi}(G_0/Q_0, \C_{\chi^{-s}}) \to C^\infty(\bar N_0, \C_{\chi^{-s}})$
is an injection.
We may define the action 
of $\Cal{U}(\fg)$ on the image of the restriction map by 
$\Pi_s(u) \acts f = \big(\Pi_s(u) \acts F\big)|_{\bar N_0}$ for $u \in \Cal{U}(\fg)$ 
and $F \in C^\infty_{\chi}(G_0/Q_0, \C_{\chi^{-s}})$ with $f = F|_{\bar N_0}$. 
Define a right action
$R$ of $\Cal{U}(\bar \fn)$ on $C^\infty(\bar N_0, \C_{\chi^{-s}})$ by
\beu
\big( R(X) \acts f \big)(\nbar) = \frac{d}{dt} f\big( \nbar \exp(tX) \big) \big|_{t=0}
\end{equation*}
\vskip 0.1in

\noindent for $X \in \bar\fn_0$ and $f \in C^\infty(\bar N_0 ,\C_{\chi^{-s}})$. 
A direct computation shows that
\begin{equation}\label{Eqn21}
\big(\Pi_s(Y) \acts f \big)(\nbar) = -sd\chi \big( (\Ad(\nbar^{-1})Y)_\fq \big) f(\nbar) 
-\big( R \big( (\Ad(\nbar^{-1})Y)_{\bar \fn} \big) \acts f \big)(\nbar)
\end{equation}
\vskip 0.1in

\noindent for $Y \in \fg$ and $f$ in the image of the restriction map 
$C^\infty_{\chi}(G_0/Q_0, \C_{\chi^{-s}}) \to C^\infty(\bar N_0, \C_{\chi^{-s}})$. 
This equation implies that the representation $\Pi_s$ extends 
to a representation of $\Cal{U}(\fg)$ on 
the whole space $C^\infty(\bar{N}_0 , \C_{\chi^{-s}})$. 
Note that for all $Y \in \fg$,
the linear map $\Pi_s(Y)$ is in $C^\infty(\bar{N}_0) \oplus \eX(\bar{N}_0)$.
This property of $\Pi_s(Y)$ makes $\bar{N}_0$ a $\fg$-manifold. 

Let $\Cal{L}_{-s}$ be the trivial bundle of $\bar{N}_0$ with fiber $\C_{\chi^{-s}}$.
Then the space of smooth sections of $\Cal{L}_{-s}$
is identified with $C^\infty(\bar{N}_0 , \C_{\chi^{-s}})$. 
For $Y \in \fg$ and $f \in C^\infty(\bar{N}_0)$, 
a computation shows that in $\mathbb{D}(\Cal{L}_{-s})$,
\beu
\big( [\Pi_s(Y),f] \big)(\nbar) 
=-\big( R \big( (\Ad(\nbar^{-1})Y)_{\bar \fn} \big) \acts f \big)(\nbar).
\end{equation*}
This verifies that $\Pi_s$ gives $\Cal{L}_{-s}$ the structure of a $\fg$-bundle.

Now we define 
\beu
\mathbb{D}(\Cal{L}_{-s})^{\bar \fn}
=\{ D \in \mathbb{D}(\Cal{L}_{-s}) \; | \; [\Pi_s(X), D] = 0 
\text{ for all $X \in \bar \fn$} \}.
\end{equation*}
\vskip 0.1in

\begin{Prop}\cite[Proposition 13]{BKZ09}\label{Prop22}
Let $D_1, \ldots, D_m$ be a list of operators in $\mathbb{D}(\Cal{L}_{-s})^{\bar \fn}$.
Suppose that the list is linearly independent at $e$ and that there is a map
$b: \fg \to \f{gl}(m,\C)$ such that

\beu
\big([\Pi_s(Y), D_i] \acts f\big)(e) = \sum_{j=1}^m b(Y)_{ji}(D_j \acts f)(e)
\end{equation*} 
\vskip 0.1in
\noindent for all $Y \in \fg, \; f \in C^\infty(\bar{N}_0, \C_{\chi^{-s}})$, and $1\leq i \leq m$. 
Then $D_1, \ldots, D_m$ is a conformally invariant system on $\Cal{L}_{-s}$.
The structure operator of the system is given by 
$C(Y)(\nbar) = b(\emph{\Ad}(\nbar^{-1})Y)$ for all $\nbar \in \bar{N}_0$ and $Y \in \fg$.
\end{Prop}
As shown on pp. 801-802 in \cite{BKZ09} the differential operators in 
$\mathbb{D}(\Cal{L}_{-s})^{\bar \fn}$ can be described in terms of 
elements of the generalized Verma module
\beu
\eM_\fq(\C_{sd\chi}) = \Cal{U}(\fg) \otimes_{\Cal{U}(\fq)}\C_{sd\chi},
\end{equation*}
where $\C_{sd\chi}$ is the $\fq$-module derived from 
the $Q_0$-representation $(\chi^{s}, \C)$. 
By identifying $\eM_\fq(\C_{sd\chi}) $ as $\Cal{U}(\bar\fn) \otimes \C_{sd\chi}$, 
the map $\eM_\fq(\C_{sd\chi})  \to \Cal{U}(\bar \fn)$ given by 
$u \otimes 1 \mapsto u$ is an isomorphism. 
The composition
\begin{equation}\label{Eqn23}
\eM_\fq(\C_{sd\chi})  \to \Cal{U}(\bar \fn) \to \mathbb{D}(\Cal{L}_{-s})^{\bar \fn}
\end{equation}
is then a vector-space isomorphism, where 
the map $\Cal{U}(\bar \fn) \to \mathbb{D}(\Cal{L}_{-s})^{\bar \fn}$ is given by
$u \mapsto R(u)$.

Suppose that $f \in C^\infty(\bar{N}_0,\C_{\chi^{-s}})$ and $l \in L_0$. Then we define 
an action of $L_0$ on $C^\infty(\bar{N}_0,\C_{\chi^{-s}})$ by
\beu
(l \cdot f)(\nbar) = \chi^{-s}(l)f(l^{-1}\nbar l).
\end{equation*} 
This action agrees with the action of $L_0$ by left translation on the image 
of the restriction map $C^\infty_{\chi}(G_0/Q_0, \C_{\chi^{-s}}) \to 
C^\infty(\bar N_0, \C_{\chi^{-s}})$.
In terms of this action we define an action of $L_0$ on $\mathbb{D}(\Cal{L}_{-s})$ by
\beu
(l \cdot D) \acts f = l \cdot \big(D \acts (l^{-1} \cdot f)\big).
\end{equation*}
\noindent One can check that we have $l \cdot R(u) = R(\Ad(l)u)$ for $l \in L_0$ and 
$u \in \Cal{U}(\bar \fn)$; in particular, this $L_0$-action stabilizes the subspace
$\mathbb{D}(\Cal{L}_{-s})^{\bar \fn}$. We define an action of $L_0$ 
on $\eM_\fq(\C_{sd\chi}) $ by $l \cdot (u \otimes z) = \Ad(l)u \otimes z$.
With these actions,
the isomorphism (\ref{Eqn23}) is $L_0$-equivariant.
For $D\in \mathbb{D}(\Cal{L}_{-s})$, we denote by $D_{\nbar}$ the linear functional
$f \mapsto (D\acts f)(\nbar)$ for $f \in C^\infty(\Nbar_0, \C_{\chi^{-s}})$. 
The following result is the specialization of Theorem 15 in \cite{BKZ09}
to the present situation.

\begin{Thm}\label{Thm24}
Suppose that $F$ is a finite-dimensional $\fq$-submodule of 
the generalized Verma module $\eM_\fq(\C_{sd\chi}) $. Let $f_1, \ldots, f_k$ be a basis of $F$
and define constants $a_{ri}(Y)$ by
\beu
Yf_i = \sum_{r=1}^k a_{ri}(Y)f_r
\end{equation*}
for $1\leq i \leq k$ and $Y \in \fq$. Let $D_1, \ldots, D_k \in 
\mathbb{D}(\Cal{L}_{-s})^{\bar \fn}$ correspond to the elements 
$f_1, \ldots, f_k \in F$. Then
for all $Y \in \fg$, $1\leq i \leq k$, and $\nbar \in \Nbar_0$,
\beu
[\Pi_s(Y), D_i]_{\nbar} 
= \sum_{r=1}^k a_{ri}\big( (\emph{Ad}(\nbar ^{-1})Y)_\fq\big) (D_r)_{\nbar}
-sd\chi\big( (\emph{Ad}(\nbar^{-1})Y)_\fq\big) (D_i)_{\nbar}.
\end{equation*}
\end{Thm}

\section{The $\Omega_3^t$ system}\label{Section4}

Let $G$ be a complex simple Lie group with Lie algebra $\fg$ simply laced.
In this section we specialize to the situation where $G_0$ is a real form
of $G$ that contains a real parabolic subgroup of Heisenberg type.
In this setting, we construct a system of differential operators
over the line bundle $\Cal{L}_{-s}$ and show some technical facts
that will be used later sections. We first introduce some notation.

Choose a Cartan subalgebra $\fh$ of $\fg$
and let $\gD$ be the set of roots of $\fg$ with respect to $\fh$. 
Fix $\gD^+$ a positive system and denote by $S$ 
the corresponding set of simple roots.
Write $\rho$ for half the sum of the positive roots.
We denote the highest root by $\gamma$.
Let $B_\fg$ denote a positive multiple of the Killing form on $\fg$ and denote
by $(\cdot, \cdot)$ the corresponding inner product induced on $\fh^*$.
The normalization of $B_\fg$ will be specified below.
Let us write $||\ga||^2 = (\ga, \ga)$ for any $\ga \in \gD$.
For $\ga \in \gD$, we let $\fg_{\ga}$ 
be the root space of $\fg$ corresponding to $\ga$.
For any $\ad(\fh)$-invariant subspace $V \subset \fg$,
we denote by $\gD(V)$ the set of roots $\ga$ so that $\fg_\ga \subset V$.

It is known that we can choose $X_\ga \in \fg_\ga$ and $H_\ga \in \fh$ for each $\ga \in \gD$
in such a way that the following conditions (C1)-(C5) hold. 
The reader may want to note that our normalizations are special cases of
those used in \cite{BKZ08}.
(C1) For each $\ga \in \gD^+$,  $\{X_{-\ga}, H_{\ga},X_\ga \}$ is an $\f{sl}(2)$-triple.
In particular, 
$[X_\ga, X_{-\ga}] = H_\ga$.
(C2) For each $\ga, \gb \in \gD$, $[H_\ga, X_\gb] = \gb(H_\ga)X_\gb$.
(C3) For $\ga \in \gD$ we have $B_\fg(X_\ga, X_{-\ga})$ = 1; in particular,
$(\ga, \ga) = 2$.
(C4) For $\ga,\gb \in \gD$ we have 
$\gb(H_\ga) = (\gb,\ga)$. 
(C5) If $\ga$, $\gb$, $\ga+\gb \in \gD$ then there is a non-zero integer 
$N_{\ga, \gb}$ so that $[X_\ga, X_\gb] = N_{\ga, \gb}X_{\ga+\gb}$.
For $Z \in \fl$ and $\ga \in \gD(V^+)$, we define a scalar $M_{\ga, \gb}(Z)$ by
$[Z, X_\ga] = \sum_{\gb \in \gD(V^+)}M_{\ga,\gb}(Z)X_\gb$.

Let $\fq=\fl \oplus \fn$ be the standard parabolic subalgebra of $\fg$ of Heisenberg type
with $\fl$ its Levi factor and $\fn$ its nilpotent radical.
Then, as stated in Introduction,
the action of $\ad(H_\gamma)$ on $\fg$ 
induces the eigenspace decomposition $(\ref{Intro1})$ of $\fg$,
where $\gamma$ is the highest root of $\fg$.
Since $\fz(\fn)=\fg_{\gamma}$ is one-dimensional, 
there is a character $\chi$ of $L_0$ so that 
$\Ad(l)X_{\gamma} = \chi(l)X_\gamma$ for all $l \in L_0$.
Note that $\Ad(l)X_{-\gamma} = \chi(l)^{-1}X_{-\gamma}$
for all $l \in L_0$, as $\fg_{-\gamma}$ is the $B_\fg$-dual space 
of $\fg_\gamma$.
For the rest of this paper,
we fix $\chi$ so that its differential $d\chi$ is $d\chi = \gamma$.

Let $\eD_\gamma(\fg, \fh)$ be the deleted Dynkin diagram 
associated to the Heisenberg parabolic $\fq$; 
that is, the subdiagram of the Dynkin diagram of $(\fg, \fh)$ obtained by deleting
the node corresponding to the simple root that is not orthogonal to $\gamma$,
and the edges that involve it.

As on p.789 in \cite{BKZ08} the operator $\Omega_2$ is given in terms of $R$ by  
\begin{equation}\label{Eqn41}
\Omega_2(Z) =  -\frac{1}{2}\sum_{\ga, \gb \in \gD(V^+)}
N_{\gb, \gb'}M_{\ga, \gb'}(Z) R(X_{-\ga})R(X_{-\gb})
\end{equation}
\noindent for $Z \in \fl$.
One can check that we have 
$\Omega_2(\Ad(l)Z) = \chi(l) l \cdot \Omega_2(Z)$
for all $l \in L_0$.
Note that this 
is different from the $\Ad(l)$ transformation law of $\Omega_2$
that appears in \cite{BKZ08},
because the parabolic $\fq$ is chosen in this paper, 
while the opposite parabolic 
$\bar \fq$ is chosen in \cite{BKZ08}.
We extend the $\C$-linear maps $d\chi$, $R$, and $\Omega_2$ to be left $C^\infty(\bar N_0)$-linear so that certain relationships can be expressed more easily.

Now for $t \in \C$ we define 
an operator $\Omega^t_3: V^- \to \mathbb{D}(\Cal{L}_{-s})^{\bar \fn}$ by
\begin{equation*}
\Omega_3^t(Y) = \tilde{\Omega}_3(Y) + tC_3(Y),
\end{equation*}
where the operators $\tilde{\Omega}_3(Y)$ and $C_3(Y)$  
are defined in terms of $R$ and $\Omega_2$ by
\begin{equation*}
\tilde{\Omega}_3(Y) = \sum_{\ge \in \gD(V^+)} R(X_{-\ge})\Omega_2\big( [X_\ge, Y] \big)
\end{equation*}
and
\begin{equation*}
C_3(Y) = R(Y)R(X_{-\gamma})
\end{equation*}
as on p.801 of \cite{BKZ08}. 

\begin{Lem}\label{Lem30}
Let $W_1, \ldots, W_m$ be a basis for $V^+$ and 
$W_1^*, \ldots, W_m^*$ be the $B_\fg$-dual basis of $V^-$.
Then
\begin{equation*}
\tilde{\Omega}_3(Y) = \sum_{i=1}^m R(W_i^*)\Omega_2\big( [W_i, Y] \big).
\end{equation*}
\end{Lem}

\begin{proof}
Suppose that $\gD(V^+) = \{ \ge_1, \ldots, \ge_m\}$.
Each $W_i$ then may be expressed by 
$W_i = \sum_{j = 1}^m a_{ij}X_{\ge_j}$
for $a_{ij} \in \C$.
Let $[a_{ij}]$ be the change of basis matrix 
and set $[b_{ij}] = [a_{ij}]^{-1}$. 
Then define 
$W^*_i = \sum_{k=1}^mb_{ki}X_{-\ge_k}$
for $i=1,\ldots, m$. 
Note that $\sum_{s=1}^m a_{is}b_{sj} = \gd_{ij}$
with $\gd_{ij}$ the Kronecker delta.
Since $B_\fg(X_{\ge_i}, X_{-\ge_j}) = \gd_{ij}$,
it follows that 
$B_\fg(W_i, W_j^*) = \gd_{ij}$. 
Thus $\{W_1^*, \ldots, W_m^*\}$
is the dual basis of $\{W_1,\ldots, W_m\}$. 
Hence,
\begin{equation*}
\sum_{i=1}^m R(W_i^*)\Omega_2([W_i, Y])
=\sum_{j,k=1}^m\big( \sum_{i=1}^m b_{ki}a_{ij}\big)R(X_{-\ge_k})\Omega_2([X_{\ge_j},Y])\\
= \sum_{j=1}^m R(X_{-\ge_j})\Omega_2([X_{\ge_j}, Y]).
\end{equation*}
\end{proof}

\begin{Lem}\label{Lem31}
For all $l \in L_0$, $Z\in \fl$, and $Y \in V^-$, we have
\begin{equation*}\label{Eqn32}
\Omega^t_3(\emph{\Ad}(l)Y) 
= \chi(l) l \cdot \Omega^t_3(Y)
\end{equation*} 
and
\beu
\Omega^t_3\big([Z,Y]\big) 
=  d\chi(Z)\Omega^t_3(Y) + [\Pi_s(Z), \Omega^t_3(Y)].
\end{equation*}
\vskip 0.1in
\end{Lem}

\begin{proof}
To obtain the first equality
it suffices to show that $\tOmega_3$ and $C_3$ have the proposed 
transformation law. 
Recall that $l \cdot R(u) = R(\Ad(l)u)$ 
for $l \in L_0$ and $u \in \Cal{U}(\bar\fn)$; in particular, 
we have $l \cdot R(X_{-\gamma}) = \chi(l)^{-1}R(X_{-\gamma})$. 
Therefore $\chi(l)l\cdot C_3(Y) = R(\Ad(l)Y)R(X_{-\gamma})$,
which is $C_3(\Ad(l)Y)$.
Note that since we have $\Omega_2(\Ad(l)W) = \chi(l)l\cdot \Omega_2(W)$ 
for $l \in L_0$ and $W\in \fl$,
it follows that
\begin{equation}\label{Eqn33}
\chi(l) l \cdot \tOmega_3(Y)
= \sum_{\ge \in \gD(V^+)} R(\Ad(l)X_{-\ge})\Omega_2\big( [\Ad(l)X_\ge, \Ad(l)Y] \big).
\end{equation}
By Lemma \ref{Lem30}, the value of $\tilde{\Omega}_3(Y)$ is independent 
from a choice of a basis for $V^+$. Therefore
the right hand side of (\ref{Eqn33}) is equal to the sum
$\sum_{\ge \in \gD(V^+)} R(X_{-\ge})\Omega_2\big( [X_\ge, \Ad(l)Y] \big)$,
which is $\tOmega_3(\Ad(l)Y)$. 
The second equality is obtained by differentiating the first.
\end{proof}

Let $\omega_3^t(Y)$ denote the element in 
$\Cal{U}(\bar \fn) \otimes \C_{sd\chi}$ that corresponds to
$\Omega^t_3(Y)$ under $R$.

\begin{Lem}\label{Lem42}
For $Z \in \fl$ and $Y \in V^-$, we have
\beu
\omega^t_3\big( [Z, Y] \big) 
= Z \omega^t_3(Y) + (1-s)d\chi(Z)\omega^t_3(Y).
\end{equation*}
\end{Lem}

\begin{proof}
Lemma \ref{Lem31} shows that $\Omega^t_3(\Ad(l)Y) = \chi(l)l\cdot \Omega^t_3(Y)$ for $l \in L_0$. Thus it follows from Lemma 18 in
\cite{BKZ09} that 
we have $\omega^t_3(\Ad(l)Y) = \chi(l)\Ad(l)\omega^t_3(Y)$.
The formula is then obtained by replacing $l$ by $\exp(tZ)$ with $Z \in \fl_0$,
differentiating, and setting at $t = 0$.
\end{proof}

Let $E$ be an irreducible $L_0$-submodule of $V^-$.
We say that the $\Omega_3\big|_E$ system exists 
if there exist $t_0, s_0 \in \C$
so that the list of differential operators
$\Omega^{t_0}_3\big|_E = 
\Omega^{t_0}_3(X_{\gb_1}), \ldots, \Omega^{t_0}_3(X_{\gb_m})$ with
$\gD(E) = \{\gb_1, \ldots, \gb_m\}$
is conformally invariant over $\Cal{L}_{-s_0}$.
Set $F_t(E) = 
\text{span}_\C\{\omega_3^t(Y) \; | \; Y \in E\}$.

\begin{Lem}\label{Lem43}
If the $\Omega^{t}_3 \big|_E$ system is conformally invariant 
for $t=t_0$ over $\Cal{L}_{-s_0}$ then 
$\fn$ acts on $F_{t_0}(E)$ trivially.
\end{Lem}

\begin{proof}
Since the $\Omega^{t_0}_3\big|_E$ system is 
conformally invariant over the line bundle $\Cal{L}_{-s_0}$,
the space $F_{t_0}(E)$ is a $\fq$-submodule of $\eM_\fq(\C_{s_0d\chi})$.
By applying Lemma \ref{Lem42} with $Z = H_\gamma$, 
we obtain $H_\gamma \omega^{t_0}_3(Y) 
= (2s_0-3)\omega^{t_0}_3(Y)$ for all $Y \in E$.
For $U \in V^+$ we have
$H_\gamma U\omega^{t_0}_3(Y) = (2s_0-2)U\omega^{t_0}_3(Y)$, and 
$H_\gamma X_\gamma \omega^{t_0}_3(Y) 
= (2s_0-1)X_\gamma \omega_2(Y)$ for all $Y \in E$.
Therefore if $U \in \fn$ then
$U\omega^{t_0}_3(Y) = 0$ for all $Y \in E$,
because otherwise $U\omega^{t_0}_3(Y)$ would have the wrong $H_\gamma$-eigenvalue to lie in $F_{t_0}(E)$.
\end{proof}

By using the transformation law 
$\omega^t_3(\Ad(l)Y) = \chi(l)\Ad(l)\omega^t_3(Y)$
for $l \in L_0$ and $Y \in V^-$,
one can check that for any $s \in \C$ the vector space isomorphism
\begin{equation}\label{Eqn45}
E \otimes \C_{(s-1)d\chi} \to F_t(E),
\end{equation}
given by $Y \otimes 1 \mapsto \omega^t_3(Y)$,
is $L_0$-equivariant with respect to the standard action of $L_0$ on
the tensor products $E \otimes \C_{(s-1)d\chi}$ and 
$F_t(E)\subset \Cal{U}(\bar \fn) \otimes \C_{sd\chi}$.
In particular, the $L_0$-module $F_{t_0}(E)$ is irreducible.
Here one may want to notice that the $L_0$-action on
$F_t(E)$ is given by $l \cdot (u \otimes 1) = \chi^s(l) (\Ad(l)u \otimes 1)$,
which is different from the one that is used to establish 
the $L_0$-equivariant isomorphism (\ref{Eqn23}).

Now suppose that the $\Omega^t_3\big|_E$ system is conformally invariant
for $t = t_0$ over $\Cal{L}_{-s_0}$. 
Then $F_{t_0}(E)$ is a $\fq$-submodule of 
$\Cal{U}(\bar \fn) \otimes \C_{s_0d\chi}$.
Since $F_{t_0}(E)$ is an irreducible $L_0$-module 
and $\fn$ acts on it trivially by Lemma \ref{Lem43},
 the inclusion map $F_{t_0}(E) \hookrightarrow \eM_\fq(\C_{s_0d\chi})$ induces
a non-zero $\Cal{U}(\fg)$-homomorphism
of generalized Verma modules
\begin{equation*}
\eM_\fq(F_{t_0}(E)) \to \eM_\fq(\C_{s_0d\chi}),
\end{equation*}
that is given by
$u \otimes \omega_3^{t_0}(Y) \mapsto u \cdot \omega_3^{t_0}(Y)$.
In particular, the two Verma modules
$\eM_\fq(F_{t_0}(E))$
and $\eM_\fq(\C_{s_0d\chi})$
have the same infinitesimal characters.
Since we choose character $\chi$ so that $d\chi = \gamma$,
this implies that if $\varpi$ is the highest weight for $E$
then 
\begin{equation}\label{Eqn47}
||\varpi +(s_0 - 1)\gamma + \rho||^2 = ||s_0 \gamma + \rho||^2,
\end{equation}
This will restrict the possibility of $s_0$ for which 
the $\Omega_3^t$ is conformally invariant.

\section{The $\Omega_3$ system on $\f{sl}(3,\C)$}\label{Section5}

In this short section we take a complex Lie group $G$ 
in  Section \ref{Section4} to be $SL(3,\C)$ and show
that the $\Omega_3$ system(s) exists over the line bundle $\Cal{L}_0$.
Since the generalized Verma module 
$\eM_\fq(\C_{sd\chi})$ is a (ordinary) Verma module in this case, 
we simply write $\eM(\C_{sd\chi})=\eM_\fq(\C_{sd\chi})$ 
throughout this section.

Let $\ga_1$ and $\ga_2$ be the two simple roots for $\f{sl}(3,\C)$.
Then we have $V^- = \C X_{-\ga_1} \oplus \C X_{-\ga_2}$;
each of $\C X_{-\ga_i}$ for $i = 1,2$ is an $L_0$-submodule of $V^-$.
Note that a direct computation shows that
$\Omega_3^t(X_{-\ga_i}) = \tOmega_3(X_{-\ga_i}) + t C_3(X_{-\ga_i})$
is not identically zero for $i=1,2$ and for any $t \in \C$.
Then by solving (\ref{Eqn47}) with $\varpi = -\ga_i$ for $i=1,2$,
one can see that if $\Omega^t_3(X_{-\ga_i})$ is conformally invariant 
over $\Cal{L}_{-s_0}$ then the special value $s_0$ of $s$ must be
$s_0 = 0$. Now we show that there exists a unique $t_i \in \C$ so that 
$\Omega^{t_i}_3(X_{-\ga_i})$ is conformally invariant over $\Cal{L}_0$.

\begin{Thm}\label{Thm51}
Let $\fg$ be the complex simple Lie algebra of type $A_2$, and $\fq$ be
the parabolic subalgebra of Heisenberg type. Then for each $i=1,2$
the operator $\Omega^{t}_3(X_{-\ga_i})$ is conformally invariant over $\Cal{L}_0$
if and only if $t = \frac{3}{4}$.
\end{Thm}

\begin{proof}
Fix $\ga_i$ and denote by $\ga_k$ the other simple root 
so that $S = \{\ga_i, \ga_k\}$. 
Observe that $\omega^t_3(X_{-\ga_i})$
is the element in $\eM(\C_0)$ that corresponds to 
$\Omega^t_3(X_{-\ga_i})$ 
in $\mathbb{D}(\Cal{L}_0)^{\bar \fn}$ under the map (\ref{Eqn23}).
By Theorem \ref{Thm24} and Lemma \ref{Lem43}, it suffices to show that 
$\C \omega_3^{t}(X_{-\ga_i})$ is a $\fq$-submodule 
of $\eM(\C_0)$ with trivial $\fn$ action if and only if $t = \frac{3}{4}$.

A direct computation shows that the element in $\eM(\C_0)$ that corresponds
to $\tOmega_3(X_{-\ga_i})$ may be written as
\begin{equation*}
-\frac{3}{2}N_{\ga_i, \ga_k}X_{-\ga_i}^2X_{-\ga_k} \otimes 1
-\frac{3}{4}X_{-\ga_i}X_{-\gamma} \otimes 1.
\end{equation*}
As $C_3(X_{-\ga_i}) = R(X_{-\ga_i})R(X_{-\gamma})$,
the element in $\eM(\C_0)$ corresponding to $C_3(X_{-\ga_i})$ is 
$X_{-\ga_i}X_{-\gamma} \otimes 1$. 
Thus $\omega_3^t(X_{-\ga_i})$ may be given by
\begin{equation*} \label{Eqn4.2}
\omega_3^t(X_{-\ga_i}) 
= -\frac{3}{2}N_{\ga_i, \ga_k}X_{-\ga_i}^2X_{-\ga_k} \otimes 1
+\left(t-\frac{3}{4}\right)X_{-\ga_i}X_{-\gamma} \otimes 1.
\end{equation*}
One can easily check that $\fn$ acts trivially on
$\C X_{-\ga_1}^2X_{-\ga_2} \otimes 1$ and $\C X_{-\ga_2}^2X_{-\ga_1} \otimes 1$
and thus both of them are
one-dimensional $\fq$-submodules of $\eM(\C_0)$,
while it acts nontrivially on $X_{-\ga_1}X_{-\gamma}\otimes 1$ 
and $X_{-\ga_2}X_{-\gamma} \otimes 1$ in $\eM(\C_0)$. 
Therefore $\C \omega_3^t(X_{-\ga_i})$ is a $\fq$-submodule with 
trivial $\fn$ action if and only if $t = \frac{3}{4}$.
\end{proof} 

\section{The $\Omega_3$ system on $\f{so}(8,\C)$}\label{Section6}

In this section we take a complex Lie group $G$ in Section \ref{Section4}
to be $SO(8,\C)$ and show that the $\tOmega_3$ system
is conformally invariant over the line bundle $\Cal{L}_{1}$.

Note that since in this case the parabolic $\fq$ is maximal, 
the $\fl$-module $V^-$ is irreducible with highest weight $-\ga_\gamma$,
where $\ga_\gamma$ is the simple root that is not orthogonal to $\gamma$.
Then by solving (\ref{Eqn47}) with $\varpi = -\ga_\gamma$,
one can see that if the $\Omega_3$ system exists 
then the special value $s_0$ of $s$
must be $s_0 = -1$.
Thus in the rest of this paper the line bundle $\Cal{L}_{-s}$ is assumed to be $\Cal{L}_{1}$,
and for simplicity we write $\Pi$ for the 
Lie algebra action $\Pi_s$ defined in (\ref{Eqn31}) for $s = -1$.
As stated in Section \ref{Section3},
for $D\in \mathbb{D}(\Cal{L}_{-s})$, we denote by $D_{\nbar}$ the linear functional
$f \mapsto (D\acts f)(\nbar)$ for $f \in C^\infty(\Nbar_0, \C_{\chi^{-s}})$. 

\begin{Prop}\label{Prop34}
For all $X \in \fg$, $Y \in V^{-}$, and $\bar{n} \in \bar{N}_0$, we have
\beu
[\Pi(X), R(Y)]_{\nbar} = 
R\big( [\emph{Ad}(\nbar^{-1})X, Y]_{V^-} \big)_{\bar{n}} 
- d\chi \big( [\emph{Ad}(\nbar^{-1})X, Y]_{\fl} \big).
\end{equation*}
\end{Prop}

\begin{proof}
Let $F$ be the subspace of $\eM_\fq(\C_{-d\chi}) $ spanned by $X_{-\ga} \otimes 1$ 
and $1\otimes1$ with  $\ga \in \gD(V^+)$. A direct computation shows that 
$F$ is a $\fq$-submodule of $\eM_\fq(\C_{-d\chi}) $ and that for $Z \in \fl$ and $U \in \fn$ we have
\beu
Z(X_{-\ga} \otimes 1) = [Z, X_{-\ga}]\otimes 1 - d\chi(Z)X_{-\ga}\otimes 1
\end{equation*}
and
\beu
U(X_{-\ga} \otimes 1) = -d\chi([U, X_{-\ga}]_\fl)1\otimes 1.
\end{equation*}
Then it follows from Theorem \ref{Thm24} that if
$X \in \fg$ and $\big( \Ad(\nbar^{-1})X\big)_\fq = Z + U$ 
with $Z\in \fl$ and $U \in \fn$ then for $Y \in V^-$,
\beu
[\Pi(X), R(Y)]_{\nbar} = R\big([Z,Y]\big)_{\nbar} - d\chi\big([U,Y]\big).
\end{equation*}
Since $[Z,Y] = [\Ad(\nbar^{-1})X, Y]_{V^-}$ 
and $[U, X_{-\ga}]_\fl = [\Ad(\nbar^{-1})X, Y]_{\fl}$, this completes the proof.
\end{proof}

\vskip 0.1in

Let $\omega_2(W)$ denote the element in $\Cal{U}(\bar \fn)\otimes \C_{-d\chi}$
that corresponds to $\Omega_2(W)$ under $R$.  
Observe that we have
$\Omega_2(\Ad(l)W) = \chi(l) l \cdot \Omega_2(W)$
for all $l \in L_0$,
that is the same $\Ad(l)$ transformation law as $\Omega^t_3$ 
(see Lemma \ref{Lem31}).
It then follows from Lemma \ref{Lem42} with $s = -1$
that for $W, Z \in \fl$, we have
\begin{equation}\label{Eqn62}
\omega_2\big( [Z, W] \big) = Z \omega_2(W) + 2d\chi(Z)\omega_2(W).
\end{equation}

\begin{Prop}\label{Prop35}
For all $X \in \fg$, $W \in \fl$, and $\nbar \in \bar{N}_0$, we have 
\beu
[\Pi(X), \Omega_2(W) ] _{\nbar}
= \Omega_2\big([\emph{Ad}(\nbar^{-1})X, W]_{\fl}\big)_{\nbar} 
- d\chi \big((\emph{Ad}(\nbar^{-1})X)_\fl \big)\Omega_2(W)_{\nbar}.
\end{equation*}
\end{Prop}

\begin{proof}
It follows from Theorem 5.2 of \cite{BKZ08} and 
the data tabulated in Section 8.10 of \cite{BKZ08}
that each $\Omega_2$ system associated to a singleton component of 
$\eD_\gamma(\fg,\fh)$
is conformally invariant on the line bundle $\Cal{L}_{1}$.
The reader may want to note here that
the special values of our $\Omega_2$ system are 
of the form $-s_0$ with $s_0$ the special values of the $\Omega_2$ system
given in \cite{BKZ08}, as the parabolic $\fq$ is chosen in this paper,
while the opposite parabolic $\bar \fq$ is chosen in \cite{BKZ08}.
Therefore $F \equiv \text{span}_\C\{  \omega_2(W) \; | \; W \in \fl\}$ is 
a $\fq$-submodule of $\eM_\fq(\C_{-d\chi})$. 
The same argument for the proof for Lemma \ref{Lem43} shows that 
$\fn$ acts on $F$ trivially. By (\ref{Eqn62}), we have
\beu
Z\omega_2(W) = \omega_2([Z,W]) - 2d\chi(Z)\omega_2(W)
\end{equation*}
for $Z,W \in \fl$. The proposed formula now follows from Theorem \ref{Thm24}.
\end{proof}

\begin{Lem}\label{Lem34}
For $X \in V^+$ and $Y \in V^-$, we have 
\beu
\sum_{\ge \in \gD(V^+)}\Omega_2\big([[X, X_{-\ge}],[X_\ge, Y]]\big)
=2\Omega_2([X,Y]).
\end{equation*}
\end{Lem}

\begin{proof}
Since we have $||\ge||^2 = 2$ for all $\ge \in \gD(V^+)$,
it follows from Proposition 2.2 of \cite{BKZ08} that
\beu
\sum_{\ge \in \gD(V^+)}\Omega_2\big( [[X, X_{-\ge}],[X_\ge, Y]] \big)
= \frac{1}{2} \sum_{\eC}p(D_4,\eC)\Omega_2 \big( \text{pr}_{\eC}([X,Y])\big),
\end{equation*}
where $\eC$ are the connected components of $\eD_\gamma(\fg, \fh)$
as in \cite{BKZ08} and $\text{pr}_{\eC}([X,Y])$ is the projection
of $[X,Y]$ onto $\fl(\eC)$, the ideal of $[\fl, \fl]$ corresponding to $\eC$.
(See Section 2 of \cite{BKZ08} for further discussion.)
One can find in Section 8.4 of \cite{BKZ08} that 
$p(D_4, \eC) = 4$ for all the components $\eC$. Then the fact that
$\Omega_2(H_\gamma) = 0$ shows that we obtain
\beu
\sum_{\ge \in \gD(V^+)}\Omega_2\big( [[X, X_{-\ge}],[X_\ge, Y]] \big)
= 2\Omega_2 \big( [X,Y]\big),
\end{equation*}
which is the proposed formula.
\end{proof}

Now with the above lemmas and propositions
we are ready to show the following key theorem.

\begin{Thm}\label{Thm36}
We have $[\Pi(X), \tilde{\Omega}_3(Y)]_e = 0$ for all $X \in V^+$ and all $Y \in V^-$.
\end{Thm}

\begin{proof}
Observe that 
$\tilde{\Omega}_3(Y) = \sum_{\ge \in \gD(V^+)} R(X_{-\ge})\Omega_2\big( [X_\ge, Y] \big)$.
Then the commutator $[\Pi(X), \tilde{\Omega}_3(Y)]$ is a sum of two terms.
One of them is given by
\begin{align}\label{Eqn37}
&\sum_{\ge\in\gD(V^+)}[\Pi(X), R(X_{-\ge})]\Omega_2\big( [X_\ge, Y]\big)\\
&=
\sum_{\ge\in\gD(V^+)}R\big([\Ad(\cdot^{-1})X, X_{-\ge}]_{V^-}\big) \Omega_2\big( [X_\ge, Y] \big)
-\sum_{\ge\in\gD(V^+)}
 d\chi\big( [\Ad(\cdot^{-1})X, X_{-\ge}]_\fl\big) \Omega_2\big([X_\ge, Y]\big),
\nonumber
\end{align}
by Proposition \ref{Prop34}. At $e$, the first term is zero,
since $[X, X_{-\ge}]_{V^-} = 0$ for all $\ge \in \gD(V^+)$.
By writing out $X$ as a linear combination of $X_\ga$
with $\ga \in \gD(V^+)$,
one can see that at the identity the second term in (\ref{Eqn37}) evaluates to
\begin{equation*}
-\sum_{\ge\in\gD(V^+)} d\chi\big( [X, X_{-\ge}]\big) \Omega_2\big([X_\ge, Y]\big)_e
=-\Omega_2\big([X,Y]\big)_e
\end{equation*}
since $d\chi(H_\ga)=1$ for $\ga \in \gD(V^+)$. 
The other term is given by
\begin{align}\label{Eqn38}
&\sum_{\ge\in\gD(V^+)} R(X_{-\ge})\big[\Pi(X), \Omega_2( [X_\ge, Y])\big]\\
&= \sum_{\ge\in\gD(V^+)} R(X_{-\ge})\Omega_2\big([\Ad(\cdot^{-1})X,[X_\ge, Y]]_\fl\big)
- \sum_{\ge\in\gD(V^+)}
 R(X_{-\ge})d\chi\big( (\Ad(\cdot^{-1})X)_\fl \big)\Omega_2\big([X_\ge,Y]\big), \nonumber
\end{align}
by Proposition \ref{Prop35}. 
To further evaluate this expression, we make use of a simple general observation.
Namely, if $D$ is a first order differential operator, $\phi$ and $\psi$ are smooth functions,
and $\phi(e) =0$ then $D_e(\phi\psi) = D_e(\phi)\psi(e)$.
Notice that $\nbar \mapsto \ad(\Ad(\nbar^{-1})X)$ is a smooth function on $\Nbar_0$.
It follows from the left $C^\infty(\bar N_0)$-linear extension of $\Omega_2$
that the first term of the right hand side of $(\ref{Eqn38})$ can be expressed as
\beu
\sum_{\ge \in \gD(V^+)}
R(X_{-\ge}) \big( \ad(\Ad(\cdot^{-1})X)_\fl \cdot \Omega_2\big([X_\ge, Y]\big) \big),
\end{equation*} 
where $\ad(\Ad(\cdot^{-1})X)_\fl$ denotes the map $Z \mapsto [\Ad(\cdot^{-1})X, Z]_\fl$
for $Z \in \fg$.
Since we have
\beu 
\big(R(X_{-\ge})\acts (\Ad(\cdot^{-1})X)\big)(e)  = [X,X_{-\ge}],
\end{equation*}
\vskip 0.05in
\noindent
$[X, [X_\ge, Y]]_\fl = 0$, and
$X_\fl = 0$,  
the right hand side of (\ref{Eqn38}) then evaluates at the identity to
\beu
\sum_{\ge\in\gD(V^+)}\Omega_2\big([[X, X_{-\ge}],[X_\ge, Y]]\big)_e
-\sum_{\ge\in\gD(V^+)} d\chi\big( [X, X_{-\ge}]\big) \Omega_2\big([X_\ge, Y]\big)_e,
\end{equation*}
which is equivalent to
\beu
\sum_{\ge\in\gD(V^+)}\Omega_2\big([[X, X_{-\ge}],[X_\ge, Y]]\big)_e
-\Omega_2\big([X,Y]\big)_e.
\end{equation*}
Therefore we obtain
\beu
[\Pi(X), \tilde{\Omega}_3(Y)]_e
= \sum_{\ge\in\gD(V^+)}\Omega_2\big([[X, X_{-\ge}],[X_\ge, Y]]\big)_e
-2\Omega_2\big([X,Y]\big)_e.
\end{equation*}
Now it follows from Lemma \ref{Lem34} that 
$[\Pi(X), \tilde{\Omega}_3(Y)]_e = 0$.
\end{proof}

\begin{Prop}\label{Prop37}
For $Y \in V^-$, we have $[\Pi(X_\gamma), \tilde{\Omega}_3(Y)]_e = 0$.
\end{Prop}

\begin{proof}
Since $\fz(\fn) = [V^+, V^+]$, it suffices to show that 
$[\Pi\big([X_1, X_2]\big), \tilde{\Omega}_3(Y)]_e = 0$ for $X_1, X_2 \in V^+$.
Note that we have $\Pi\big([X_1, X_2]\big) = [\Pi(X_1), \Pi(X_2)]$, so 
it follows from the Jacobi identity that
$[\Pi\big([X_1, X_2]\big), \tilde{\Omega}_3(Y)]$ may be 
expressed as a sum of two terms. The first is
\begin{equation*}
[\Pi(X_1), [\Pi(X_2), \tilde{\Omega}_3(Y)]] 
=\Pi(X_1)[\Pi(X_2), \tilde{\Omega}_3(Y)] - [\Pi(X_2), \tilde{\Omega}_3(Y)]\Pi(X_1).
\end{equation*}
\vskip 0.05in
\noindent By (\ref{Eqn21}), we have $\Pi(X)_e = 0$ for all $X \in \fn$. 
Using this fact and Theorem \ref{Thm36}, it is obtained that 
$[\Pi(X_1), [\Pi(X_2), \tilde{\Omega}_3(Y)]] _e = 0$ 
since $(D_1 D_2)_e = (D_1)_e  D_2$ for $D_1, D_2 \in \mathbb{D}(\Cal{L}_1)$.
The second term is 
\beu
[\Pi(X_2), [\tilde{\Omega}_3(Y), \Pi(X_1)]]
=\Pi(X_2)[\tilde{\Omega}_3(Y),\Pi(X_1)] - [\tilde{\Omega}_3(Y),\Pi(X_1)]\Pi(X_2).
\end{equation*}
\vskip 0.05in
\noindent It follows from the same argument for the first term that we have
$[\Pi(X_2), [\tilde{\Omega}_3(Y),\Pi(X_1)]]_e = 0$. 
This concludes the proposition.
\end{proof}

\begin{Thm}\label{Thm311}
Let $\fg$ be the complex simple Lie algebra of type $D_4$, 
and $\fq$ be the parabolic subalgebra of Heisenberg type.
Then the $\tilde{\Omega}_3$ system is conformally invariant on the line bundle $\Cal{L}_1$.
\end{Thm}

\begin{proof}
By Proposition \ref{Prop22}, it suffices to check the conformal invariance 
of $[\Pi(X), \tOmega_3(Y)]$ at the identity for all $X \in \fg$ and all $Y \in V^-$.
It follows from Lemma \ref{Lem31} that 
\beu
[\Pi(Z), \tOmega_3(Y)]_e = \tOmega_3\big([Z,Y]\big)_e - d\chi(Z)\tOmega_3(Y)_e
\end{equation*}
for all $Z \in \fl$. 
Also Theorem \ref{Thm36} and Proposition \ref{Prop37} show that
$[\Pi(U),\tOmega_3(Y)] = 0$ for all $U \in \fn$.
As $\tOmega_3(Y)$ is an element in  $\mathbb{D}(\Cal{L}_1)^{\bar \fn}$,
it is clear that $[\Pi(\bar U),\tOmega_3(Y)]_e = 0$ for all $\bar U \in \bar\fn$.
Since $\fg = \bar \fn \oplus \fl \oplus \fn$, 
this concludes that the $\tOmega_3$ system
is conformally invariant on $\Cal{L}_1$. 
\end{proof}

Theorem \ref{Thm311} implies that 
$F_0(V^-) = \text{span}_\C\{\omega_3^0(Y)\; | \; Y \in V^-\}$ is a $\fq$-submodule
of $\eM_{\fq}(\C_{-d\chi})$, where $\omega^0_3(Y)$ is the element in $\eM_{\fq}(\C_{-d\chi})$
that corresponds to $\tOmega_3(Y) = \Omega_3^0(Y)$ under $R$. 
The argument after Lemma \ref{Lem43} then shows that 
there exists a non-zero $\Cal{U}(\fg)$-homomorphism 
\begin{equation*}
\eM_\fq(F_0(V^-)) \to \eM_\fq(\C_{-d\chi}).
\end{equation*}
It follows from Lemma \ref{Lem42} that $H_\gamma$ acts on $F_0(V^-)$ by $-5$,
while it acts on $\C_{-d\chi}$ by $-2$; in particular, $F_0(V^-)$ is not equivalent to $\C_{-d\chi}$. 
We now conclude the following corollary.

\begin{Cor}
Let $\fg$ be the complex simple Lie algebra of type $D_4$, 
and $\fq$ be the parabolic subalgebra of Heisenberg type.
Then the generalized Verma module $\eM_\fq(\C_{-d\chi})$ is reducible.
\end{Cor}


\bibliography{Draft8}

\end{document}